\documentclass{amsart}
\usepackage{amsmath}
\usepackage{amssymb}
\usepackage{amsthm}
\usepackage{dsfont}
\usepackage{cite}
\usepackage{lineno}
\usepackage{subfigure}
\usepackage{graphicx}
\usepackage{multirow}
\usepackage{color}
\usepackage{xspace}
\usepackage{pdfsync}
\usepackage{hyperref} 
\usepackage{algorithmicx}
\usepackage{algpseudocode}
\usepackage{algorithm}
\usepackage{mathtools}
\usepackage{enumitem}
\usepackage{bigints}
\graphicspath{{Figures/}}

\DeclareMathOperator*{\argmin}{argmin}

\newcommand{\bq}{\begin{equation}}
\newcommand{\eq}{\end{equation}}
\newcommand{\R}{\mathbb{R}}
\newcommand{\N}{\mathbb{N}}

\newcommand{\abs}[1]{\left\vert#1\right\vert}

\newcommand{\G}{\mathcal{G}}

\newcommand{\bO}{\mathcal{O}}

\newcommand{\Dt}{\mathcal{D}}

\newcommand{\I}{\mathcal{I}}

\newcommand{\Sf}{\mathcal{S}}
\newcommand{\MA}{Monge-Amp\`ere\xspace}

\algnewcommand{\LineComment}[1]{\State \(\triangleright\) #1}

\newtheorem{theorem}{Theorem}
\theoremstyle{lemma}
\newtheorem{lemma}[theorem]{Lemma}

\newtheorem{definition}[theorem]{Definition}

\newtheorem{hypothesis}[theorem]{Hypothesis}
\theoremstyle{remark}

\makeatletter
\newcommand\appendix@section[1]{%
\refstepcounter{section}%
\orig@section*{Appendix \@Alph\c@section: #1}%
}
\let\orig@section\section
\g@addto@macro\appendix{\let\section\appendix@section}
\makeatother

\begin{document}

\title[Optimal Transport from 1D to 2D]{Numerical Optimal Transport from 1D to 2D using a Non-local Monge-Amp\`ere Equation}

\author{Matthew A. Cassini}
\address{Department of Mathematical Sciences, New Jersey Institute of Technology, University Heights, Newark, NJ 07102}
\email{mc225@njit.edu}

\author{Brittany Froese Hamfeldt}
\thanks{The second author was partially supported by NSF DMS-2308856.}
\address{Department of Mathematical Sciences, New Jersey Institute of Technology, University Heights, Newark, NJ 07102}
\email{bdfroese@njit.edu}

\begin{abstract}
We consider the numerical solution of the optimal transport problem between densities that are supported on sets of unequal dimension.  Recent work by McCann and Pass reformulates this problem into a non-local Monge-Amp\`ere type equation.  We provide a new level set framework for interpreting this non-linear PDE.  We also propose a novel discretisation that combines carefully constructed monotone finite difference schemes with a variable-support discrete version of the Dirac delta function.  The resulting method is consistent and monotone.  These new techniques are described and implemented in the setting of 1D to 2D transport, but can easily be generalised to higher dimensions.  Several challenging computational tests validate the new numerical method.
\end{abstract}

\date{\today}    
\maketitle

\section{Introduction}\label{sec:intro}

The problem of optimal transport~\cite{villani2021topics} involves finding a mapping $T:X\to Y$ that rearranges a density $f$ supported on $X$ into a density $g$ supported on $Y$ while minimising a given cost function
\bq\label{eq:ot}
\inf\limits_{T\#f = g} \int_\Omega c(x,T(x)) f(x)\,dx.
\eq

This problem has been widely studied, both theoretically and numerically, in the setting where the sets $X$ and $Y$ lie in same ambient space (eg: $X,Y \subset \R^n$).  However, in many important problems this simplifying assumption is not valid.  In this article, we consider the problem of optimal transport between densities that are supported on sets of different dimension ($X \subset \R^n$, $Y\subset \R^m$, $n<m$).  This setting is critical to applications such as economics and game theory ($m \geq n \geq 1$)~\cite{chiappori2017multi,galichon2016optimal,NennaPass}, simulation of atmosphere and ocean dynamics ($2 \leq n \leq 3 = m$)~\cite{cullen2006mathematical}, and Generalized Adversarial Networks ($m \neq n$)~\cite{lin2020gans}.  

Despite the importance of this problem, almost nothing is known about its numerical solution except in the case of a one-dimensional target ($n=1$) combined with significant extra structure on the problem data~\cite{chiappori2017multi}.  Utilizing the linear programming structure of the optimal transport problem is theoretically possible but computationally infeasible as it requires the discretization of an $(m+n)$-dimensional domain.

In this article, we propose the first PDE-based approach to the problem of optimal transportation between unequal dimensions.  The new method we propose relies on a reformulation in terms of a non-local \MA type equation that was recently proposed by 
McCann and Pass~\cite{mccannpass_otunequal}:
\bq\label{eq:PDE}
\begin{cases}
\int_{\partial_cu(x)}\det(D^2u(x) + A(x,y)) \psi(x,y)\, dH_{m-n}(y) = f(x), & x\in X\subset\R^n\\
D^2u(x) + A(x,y) \geq 0, & x \in X, y \in \partial_cu(x)\\
\partial_cu(x) \subset Y \subset \R^m, & x \in X.
\end{cases}
\eq
Here $m \geq n \geq 1$, $H_k$ denotes the $k$-dimensional Haussdorf measure, $\partial_cu$ is the $c$-subgradient of $u$, and the matrix-valued function $A$ and scalar-valued function $\psi$ encode information about the cost and density functions.

While the numerical solution of the \MA equation has received a great deal of attention in recent years~\cite{FO_MATheory,feng2021narrow,WanMA,ObermanWS,benamou2016monotone,feng2017convergent,Nochetto_MAConverge,DeanGlowinski,dutchgroup,FengNeilan,BrennerNeilanMA2D,HL_ThreeDimensions}, nothing is known about non-local extensions.  We propose a novel numerical method for~\eqref{eq:PDE} that is inspired by recent developments in the numerical analysis of fully nonlinear elliptic equations, but requires significant new ideas in order to couple the unequal dimensions.  Because the techniques we utilise are motivated by existing convergence theory, there is a great reason to hope that a proof of convergence of this method will soon be developed. 

As a proof-of-concept, the new method is implemented in the 1D to 2D case.  However, we emphasise that this method employs techniques that are designed to easily generalise to higher dimensions.  A range of computational examples demonstrate that this is an effective approach for solving a very challenging problem.

\section{Numerical Solution of Monge-Amp\`ere equations}\label{sec:ma}
While nothing is known about the numerical solution of general non-local \MA equations of the form~\eqref{eq:PDE}, there has been a great deal of progress in the simpler case resulting from optimal transport between equal dimensions ($m=n$) with quadratic cost ($c(x,y) = \frac{1}{2}\abs{x-y}^2$).  In this setting, the PDE reduces to the second boundary value problem for the \MA equation:
\bq\label{eq:MA}
\begin{cases}
g(\nabla u(x))\det(D^2u(x)) = f(x), & x \in X\\
u \text{ is convex}\\
\nabla u(X) \subset Y.
\end{cases}
\eq

The \MA equation is an example of a second-order fully nonlinear elliptic PDE.
\begin{definition}[Degenerate elliptic]\label{def:elliptic}
The operator
$F:\Omega\times\R\times\R^n\times\Sf^n\to\R$
is \emph{degenerate elliptic} if 
\[ F(x,u,p,X) \leq F(x,v,p,Y) \]
whenever $u \leq v$ and $X \geq Y$.
\end{definition}

Much of the recent progress in the design of provably convergent methods for the \MA equation is inspired by the convergence framework of Barles and Souganidis~\cite{BSnum}.  This foundational work demonstrated that a consistent, monotone, stable numerical method will converge to the weak solution of the PDE, provided the underlying PDE satisfies a strong form of comparison principle (subsolutions always lie below supersolutions).

\begin{definition}[Consistency]\label{def:consistency}
The scheme $F^h(x,u(x),u(x)-u(\cdot))$ is \emph{consistent} with the PDE
\[ F(x,u(x),\nabla u(x), D^2u(x)) = 0, \quad x\in{\Omega} \]
 if for any smooth function $\phi\in C^\infty$ and $x\in{\Omega}$,
\[ \limsup_{h\to0,y\to x, z\in\G^h\to x,\xi\to0} F^h(z,\phi(y)+\xi,\phi(y)-\phi(\cdot)) \leq F^*(x,\phi(x),\nabla\phi(x),D^2\phi(x)), 
\]
\[ \liminf_{h\to0,y\to x, z\in\G^h\to x,\xi\to0} F^h(y,\phi(y)+\xi,\phi(y)-\phi(\cdot)) \geq F_*(x,\phi(x),\nabla\phi(x),D^2\phi(x)). \]
\end{definition}

\begin{definition}[Monotonicity]\label{def:monotonicity}
The scheme~$F^h(x,u(x),u(x)-u(\cdot))$ is \emph{monotone} if $F^h$ is a non-decreasing function of its final two arguments.
\end{definition}

\begin{definition}[Stability]\label{def:stability}
The scheme~$F^h(x,u(x),u(x)-u(\cdot))$ is \emph{stable} if there exists a constant $M\in\R$, independent of $h$, such that if $u^h$ is any solution to $F^h(x,u^h(x),u^h(x)-u^h(\cdot)) = 0$ than $\|u^h\|_\infty \leq M$ for all sufficiently small $h>0$.
\end{definition}



A nice property of monotone schemes is that it is generally very easy to ensure the existence of a solution by including, at most, a small perturbation to make the scheme \emph{proper}.

\begin{definition}[Proper]
The scheme~$F^h(x,u(x),u(x)-u(\cdot))$ is \emph{proper} if there exists some $C>0$ such that if $u \geq v$ then $F^h(x,u,p)-F^h(x,v,p) \geq C(u-v)$.
\end{definition}

Any monotone scheme $F^h$ can be perturbed to a proper scheme via
\[ F^h(x,u(x),u(x)-u(\cdot)) + \epsilon^h(x) u(x) \]
where $\epsilon^h(x) \to 0$ as $h\to0$.  A typical choice, which lends itself to nice stability properties, is to let $\epsilon^h$ scale like the truncation error of $F^h$.

\begin{theorem}[Existence and uniqueness~{\cite[Theorem~8]{ObermanEP}}]\label{thm:exist}
Let $F^h$ be a continuous, proper, monotone scheme. Then $F^h(x,u(x),u(x)-u(\cdot)) = 0$ has a unique solution.
\end{theorem}

The Barles-Souganidis convergence framework has inspired many different monotone discretisations of the \MA equation~\cite{FO_MATheory,feng2021narrow,WanMA,ObermanWS,benamou2016monotone,Nochetto_MAConverge,HL_ThreeDimensions}.
Typically, these discretisations involve a PDE operator that incorporates the convexity constraint ($D^2u(x) \geq 0$) through the use of a modified determinant satisfying
\bq\label{eq:detPlus}
{\det}^+(M) = \begin{cases} \det(M), & M \geq 0\\ \leq 0, & \text{otherwise}. \end{cases}
\eq
Monotone discretisations can then be constructed for the modified \MA equation
\bq\label{eq:MAop}
F(x,u(x),\nabla u(x), D^2u(x)) \equiv -g(\nabla u(x)){\det}^+(D^2u(x)) + f(x).
\eq

The Barles-Souganidis convergence framework does not immediately apply the second boundary value problem for \MA~\eqref{eq:MA}, which does not satisfy the required comparison principle.
However, this powerful framework has recently been extended to include this important PDE~\cite{Hamfeldt_BVP2,bonnet2022monotone}.  By adding mild additional structure to the numerical methods, these works establish the convergence of consistent, monotone methods.  Moreover, these works demonstrate that as long as the interior problem is set up carefully, there is a great deal of flexibility in the choice of numerical boundary conditions.  Indeed, even boundary conditions that are inconsistent with~\eqref{eq:MA} can nevertheless yield a numerical solution that converges to the true solution in $L^\infty(\Omega)$ for any compact subset $\Omega\subset X$.

A similar convergence result was recently established by one of the authors for a class of more general \MA type equations of the form
\bq\label{eq:MAsphere}
\begin{cases}
\psi(x,\nabla u(x))\det(D^2u(x)+A(x,\nabla u(x))) = f(x), & x \in \Sf^2\\
D^2u(x)+A(x,\nabla u(x)) \geq 0\\
\end{cases}
\eq
posed on the 2-sphere $\Sf^2$~\cite{HT_OTonSphereTheory}.  In particular, this work showed that consistent, monotone methods converge with the addition of mild additional structure (which is easily incorporated into the numerical method).

With these results in mind, a reasonable starting point for the numerical solution of non-local \MA type equations~\eqref{eq:PDE} is to design consistent, monotone methods.  Given recent progress in the numerical analysis of other \MA type equations, there is great reason to hope that it will soon be possible to produce a proof of convergence for such methods.

\section{PDE Formulation}\label{sec:pde}
The starting point of our numerical method for optimal transport between unequal dimensions is the PDE formulation~\eqref{eq:PDE} proposed by McCann and Pass.
We begin this section with a formal overview of the derivation of the PDE, referring to~\cite{mccannpass_otunequal} for complete details.  We then propose a local level set formulation that highlights the elliptic structure of the equation and will lead into the numerical method introduced in \autoref{sec:method}.

To formulate the optimal transport problem, we begin with a density function $f$ supported on a set $X\subset\R^n$, a density $g$ supported on a set $Y\subset\R^m$, and a cost function $c:X\times Y \to \R$.  We require the data to satisfy the mass balance constraint
\bq\label{eq:massbalance}
\int_X f(x)\,dx = \int_Y g(y)\,dy.
\eq
We assume here that $n<m$ and seek a mapping $T^*:Y\to X$ that minimises the cost
\bq\label{eq:Tmin}
T^* = \argmin\limits_{T\#g = f} \int_Y c(T(y),y)g(y)\,dy.
\eq

We remark here that, under appropriate conditions on the data, we can hope for a well-defined mapping from the higher-dimensional space $Y$ to the lower-dimensional space.  However, we expect that any ``mapping'' from $Y$ to $X$ would have to be multi-valued as points in the lower-dimensional space must ``spread'' to fill the higher-dimensional space.

\subsection{Feasibility}
In a traditional Optimal Transport problem ($n=m$), the pushforward condition $T\#g = f$ can be expressed via the change of variables theorem as
\bq\label{eq:cov} f(T(y))\det(\nabla T(y)) = g(y) . \eq
We now seek a similar respresentation of the pushforward condition for the unequal dimension setting ($n<m$).

We begin by expressing the condition $T\#g = f$ in the form
\bq\label{eq:pushforward1}
\int_X \Phi(x)f(x)\,dx = \int_Y \Phi(T(y))g(y)\,dy
\eq
for every $\Phi\in L^1(X)$.

Now we recall the coarea formula:
\bq\label{eq:coarea}
\int_Y r(y) JT(y)\,dy = \int_X\int_{T^{-1}(x)}r(y)\,dH_{m-n}(y)\,dx
\eq
for every $r\in L^1(Y)$.  Here $H_k$ denotes the $k$-dimensional Hausdorff measure and 
\bq\label{eq:squarejac} JT(y) = \sqrt{\det\left(\nabla T(y) \nabla T(y)^T\right)}.\eq

Next we make the particular choice
\[ r(y) = \frac{\Phi(T(y))g(y)}{JT(y)} \]
to obtain
\begin{align*}
\int_X\Phi(x)f(x)\,dx &= \int_Y \Phi(T(y))g(y)\,dy\\
  &= \int_Y r(y) JT(y)\,dy \\
	&= \int_X\int_{T^{-1}(x)}\frac{\Phi(T(y)) g(y)}{JT(y)}\,dH_{m-n}(y)\,dx.
\end{align*}

We notice that if $y\in T^{-1}(x)$ then $T(y) = x$, which allows us to rewrite the final integral as
\[ \int_X\Phi(x)f(x)\,dx = \int_X \Phi(x) \int_{T^{-1}(x)} \frac{g(y)}{JT(y)}\,dH_{m-n}(y)\,dx \]
for every $\Phi\in L^1(X)$.  From this we conclude that almost everywhere,
\bq\label{eq:pushforward}
f(x) = \int_{T^{-1}(x)}\frac{g(y)}{JT(y)}\,dH_{m-n}(y).
\eq

We note that in the special case of equal dimensions ($m=n$), this coincides with the usual result of the change of variables formula~\eqref{eq:cov}.

\subsection{Optimality}
Having characterised the feasibility condition $T\#g =f$, we now seek a representation of the optimal mapping. 
In the simplest case of equal dimensions ($m=n$) and a quadratic cost ($c(x,y) = \frac{1}{2}\abs{x-y}^2$), optimality is achieved by a mapping that can be expressed as the gradient of a convex function.

 Our starting point in the current setting is the usual dual formulation of the Optimal Transport problem~\cite{villani2021topics}.  This leads us to seek a $c$-convex dual pair $u(x),v(y)$ satisfying
\[ u(x)+v(y)+c(x,y) \geq 0\]
with equality when $x = T(y)$.

Equivalently, we require the inverse optimal mapping (which may be set-valued) to be contained in the $c$-subgradient of the function $u$, which we write as
\bq\label{eq:csubgrad}
T^{-1}(x) \subset \partial_c u(x) = \{y\in Y \mid c(x,y) + u(x) = \inf_{z\in X}\{c(z,y)+u(z)\}\}.
\eq

From here, we obtain the following first- and second-order optimality conditions:
\bq\label{eq:opt1}
\nabla u(x) + \nabla_xc(x,y) = 0, \quad x = T(y),
\eq
\bq\label{eq:opt2}
D^2u(x) + D^2_{xx} c(x,y) \geq 0, \quad x = T(y).
\eq

Differentiating the first-order optimality condition~\eqref{eq:opt2}  leads to the expression
\[
D^2u(T(y)) \nabla T(y) + D^2_{xx}c(T(y),y) \nabla T(y) + D^2_{xy}c(T(y),y) = 0.
\]

Under a slightly stronger version of the second-order optimality condition~\eqref{eq:opt2} (assuming positive definiteness instead of merely semi-positive definiteness), we can obtain the following expression for the Jacobian of the optimal mapping:
\bq\label{eq:jac}
\nabla T(y) = -\left(D^2u(T(y)) + D^2_{xx}c(T(y),y)\right)^{-1} D^2_{xy}c(T(y),y).
\eq

We can now combine this with the feasibility condition~\eqref{eq:pushforward}.  We first compute the relevant Jacobian determinant~\eqref{eq:squarejac}:
\begin{align*}
JT(y) &= \sqrt{\det\left(\nabla T(y) \nabla T(y)^T\right)}\\
   &= \frac{\sqrt{\det\left( D^2_{xy}c(T(y),y) \, D^2_{xy}c(T(y),y)^T \right)}}{\det\left(D^2u(T(y)) + D^2_{xx}c(T(y),y)\right)}.
\end{align*}

Finally, we substitute this into the pushforward condition~\eqref{eq:pushforward} and simplify to obtain the non-local \MA type equation
\bq\label{eq:nonlocalMA}
\begin{split}
f(x) &= \int_{T^{-1}(x)} \frac{\det\left(D^2u(T(y)) + D^2_{xx}c(T(y),y)\right)}{\sqrt{\det\left( D^2_{xy}c(T(y),y) \, D^2_{xy}c(T(y),y)^T \right)}} g(y)\,dH_{m-n}(y)\\
  &= \int_{\partial_cu(x)}\frac{\det\left(D^2u(x) + D^2_{xx}c(x,y)\right)}{\sqrt{\det\left( D^2_{xy}c(x,y) \, D^2_{xy}c(x,y)^T \right)}} g(y)\,dH_{m-n}(y).
\end{split}
\eq

We recall that this equation is coupled to the $c$-convexity constraint~\eqref{eq:opt2},
\bq\label{eq:cconvex} D^2u(x) + D^2_{xx}c(x,y) \geq 0, \quad y \in T^{-1}(x), \eq
which plays the role of the convexity constraint ($D^2u(x) \geq 0$) in the traditional \MA equation.

\subsection{Local formulation}
As written, equations~\eqref{eq:nonlocalMA}-\eqref{eq:cconvex} are non-local since the domain of integration involves the $c$-subgradient of $u$, which relies on global information~\eqref{eq:csubgrad}.  However, in many settings it is possible to replace this global definition with a local construction derived from the first- and second-order optimality conditions~\cite{mccannpass_otunequal}.

To accomplish this, we first define sets on which the first- and second-order optimality conditions~\eqref{eq:opt1}-\eqref{eq:opt2} hold:
\bq\label{eq:Y1}
Y_1(x,\nabla u(x)) = \{y \in Y \mid \nabla u(x) + \nabla_x c(x,y) = 0\},
\eq
\bq\label{eq:Y2}
Y_2(x,\nabla u(x),D^2u(x)) = \{y \in Y_1(x,\nabla u(x)) \mid D^2u(x) + D^2_{xx}c(x,y) \geq 0\}. 
\eq

These are necessary conditions for optimality in~\eqref{eq:csubgrad}, which leads to the containment
\bq\label{eq:subsets}
\partial_c u(x)  \subseteq Y_2(x,\nabla u(x),D^2u(x)).
\eq
However, when the problem has sufficient regularity (i.e. $C^2$ solutions), this containment becomes an equality up to $H_{m-n}$-negligible sets~\cite{mccannpass_otunequal}.  In this setting,~\eqref{eq:nonlocalMA} simplifies to
\bq\label{eq:localMA}
f(x) = \int_{Y_2(x,\nabla u(x), D^2u(x))}\frac{\det\left(D^2u(x) + D^2_{xx}c(x,y)\right)}{\sqrt{\det\left( D^2_{xy}c(x,y) \, D^2_{xy}c(x,y)^T \right)}} g(y)\,dH_{m-n}(y),
\eq
subject to the $c$-convexity condition~\eqref{eq:cconvex}.

This simplification of the domain of integration reduces the problem to a local PDE.  However, given the dependence of the domain of integration on second-order information, it is not at all obvious that this equation is actually elliptic and that the arguments developed in~\autoref{sec:ma} are reasonable.

To re-express this in a way that ellucidates the elliptic structure of this equation, we recall that the $c$-convexity constraint~\eqref{eq:cconvex} ensures that
\[ \det\left(D^2u(x) + D^2_{xx}c(x,y)\right) = {\det}^+\left(D^2u(x) + D^2_{xx}c(x,y)\right)\]
where ${\det}^+$ is a modified version of the determinant given by
\[{\det}^+(M) = \begin{cases}\det(M), & M \geq 0\\0, & \text{otherwise}. \end{cases}\]

This allows us to absorb the $c$-convexity constraint into the PDE and reformulate~\eqref{eq:cconvex} and~\eqref{eq:localMA} into the single equation
\bq\label{eq:localconvexMA}
f(x) = \int_{Y_2(x,\nabla u(x), D^2u(x))}\frac{{\det}^+\left(D^2u(x) + D^2_{xx}c(x,y)\right)}{\sqrt{\det\left( D^2_{xy}c(x,y) \, D^2_{xy}c(x,y)^T \right)}} g(y)\,dH_{m-n}(y).
\eq
Following the arguments of~\cite[Lemma~2.6]{Hamfeldt_BVP2}, we emphasise that solutions of this equation are automatically $c$-convex on the domain $X$ since the density function $f$ is positive and $f$ and $g$ are required to satisfy the mass balance constraint~\eqref{eq:massbalance}.

This observation makes a further simplification possible.  We notice that for any $x\in X$ and $y\in Y_1(x,\nabla u(x))$, either $y\in Y_2(x,\nabla u(x),D^2u(x))$ (if $D^2u(x) + D^2_{xx}c(x,y) \geq 0$) or ${\det}^+\left(D^2u(x) + D^2_{xx}c(x,y)\right) = 0$.  Thus we can replace the domain of integration in~\eqref{eq:localconvexMA} with the set of points satisfying the first-order optimality conditions to obtain
\bq\label{eq:localfirstMA}
f(x) = \int_{Y_1(x,\nabla u(x))}\frac{{\det}^+\left(D^2u(x) + D^2_{xx}c(x,y)\right)}{\sqrt{\det\left( D^2_{xy}c(x,y) \, D^2_{xy}c(x,y)^T \right)}} g(y)\,dH_{m-n}(y).
\eq

Importantly, equation~\eqref{eq:localfirstMA} is now a degenerate elliptic \MA type equation since the only second-order term is ${\det}^+\left(D^2u(x) + D^2_{xx}c(x,y)\right)$, while all coefficients of this term are non-negative.

\subsection{Level set formulation}
While \MA type equations (and their numerical solution) have received a great deal of attention in recent years, it is less clear how to deal with their integration over a manifold $Y_1(x,\nabla u(x))$ that itself depends on the solution of the equation.  Here we propose a level-set formulation of~\eqref{eq:localfirstMA} that transfers all of the $u$-dependence into the integrand.


First, we note that we can easily enlarge the domain of integration by introducing for each $x\in X$ a Dirac delta function $\delta$ supported on the $m-n$ dimensional manifold $Y_1(x,\nabla u(x))$.  To make this precise, we let $d_{Y_1(x,\nabla u(x))}(y)$ denote the Euclidean distance between $y$ and the manifold $Y_1(x,u'(x))$. This allows us to rewrite~\eqref{eq:localfirstMA} as
\bq\label{eq:MAdelta}
f(x) = \int_{Y}\frac{{\det}^+\left(D^2u(x) + D^2_{xx}c(x,y)\right)}{\sqrt{\det\left( D^2_{xy}c(x,y) \, D^2_{xy}c(x,y)^T \right)}} g(y) \delta(d_{Y_1(x,\nabla u(x))}(y))\,dy.
\eq

Next, we seek an easier representation of the delta function since the distance function may not be easy to compute in practice.  
To facilitate this, for each $x\in X$ we define the $n$ level-set functions
\bq\label{eq:phi} \phi^x_i(y) = \frac{\partial u}{\partial x_i}(x) + \frac{\partial c}{\partial x_i}(x,y), \quad i = 1, \ldots, n.\eq
We notice immediately that the domain of integration in~\eqref{eq:localfirstMA} is equivalent to the intersection of the zero level-sets of the $\phi^x_i$. That is,
\bq\label{eq:zero}
Y_1(x,\nabla u(x)) = \{y\in Y \mid \phi^x_i(y) = 0,  i= 1, \ldots, n\}.
\eq

Following~\cite{cheng2000level}, we can rewrite equation~\eqref{eq:MAdelta} in terms of these level set functions, though the specific form is dependent on the smaller dimension $n$.  When $n=1$, the equation is equivalent to
\bq\label{eq:MAlevel1}
f(x) = \int_{Y}\frac{{\det}^+\left(D^2u(x) + D^2_{xx}c(x,y)\right)}{\sqrt{\det\left( D^2_{xy}c(x,y) \, D^2_{xy}c(x,y)^T \right)}} g(y) \delta(\phi^x_1(y))\abs{\nabla\phi^x_1(y)}\,dy.
\eq
When $n=2$, the PDE becomes
\bq\label{eq:MAlevel2}
f(x) = \int_{Y}\frac{{\det}^+\left(D^2u(x) + D^2_{xx}c(x,y)\right)}{\sqrt{\det\left( D^2_{xy}c(x,y) \, D^2_{xy}c(x,y)^T \right)}} g(y) \delta(\phi^x_1(y)) \delta(\phi^x_2(y))\abs{\nabla\phi^x_1(y)} \abs{P_{\nabla\phi^x_1(y)}\nabla\phi^x_2(y)}\,dy
\eq
where $P_v w$ denotes a projection of the vector $w$ onto the subspace orthogonal to the vector $v$.

For higher values of the dimension $n$, this same procedure can be continued using additional orthogonal projections.

Finally, we note that the gradients of these level set functions are given by
\bq\label{eq:gradlevel}
\nabla\phi^x_i(y) = \nabla_y \frac{\partial c}{\partial x_i}(x,y),
\eq
which do not depend on the potential function $u$.


\section{Numerical Method}\label{sec:method}
In this section, we introduce a numerical method for solving the optimal transport problem between unequal dimensions, starting from the level-set PDE formulation introduced in \autoref{sec:pde}.  For simplicity of the exposition, we will describe the method in the low-dimensional setting ($n=1$, $m=2$).  However, all of the techniques we propose can be extended to higher dimensions.

Our goal here is to produce a consistent, monotone discretisation of the ODE operator in~\eqref{eq:MAlevel1}, which now takes the form
\bq\label{eq:MA1to2}
-\int_{Y}{\left(u''(x) + \frac{\partial^2c}{\partial x^2}(x,y)\right)^+}\psi(x,y)\delta\left(u'(x)+\frac{\partial c}{\partial x}(x,y)\right)\,dy + f(x).
\eq
Here we let
\bq\label{eq:psi}
\psi(x,y) = \frac{g(y)\abs{\nabla_y\frac{\partial c}{\partial x}(x,y)}}{\sqrt{\left(\frac{\partial^2c}{\partial x\partial y_1}(x,y)\right)^2 + \left(\frac{\partial^2c}{\partial x\partial y_2}(x,y)\right)^2}},
\eq
which is a known function with no dependence on the potential $u$.

Since to date little is known about this equation except in the classical setting, we will focus our numerical analysis on the setting in which $C^2$ solutions exist.  However, we expect that with ongoing development of the theory of weak solutions, we will be able to extend the numerical analysis to the non-smooth setting using the techniques of~\cite{Hamfeldt_BVP2}.

In particular, we make the following assumptions.
\begin{hypothesis}[Hypotheses on problem]\label{hyp}
\end{hypothesis}
\begin{enumerate}
\item[(H1)] The cost function $c\in C^4(\bar{X}\times\bar{Y})$.
\item[(H2)] The mixed partial derivatives of the cost $\sqrt{\left(\frac{\partial^2c}{\partial x\partial y_1}(x,y)\right)^2 + \left(\frac{\partial^2c}{\partial x\partial y_2}(x,y)\right)^2}$ are bounded away from zero.
\item[(H3)] The positive density functions $f\in C^0(\bar{X})$, $g\in C^2(\bar{Y})$ are bounded away from zero and infinity.
\item[(H4)] The solution of~\eqref{eq:MA1to2} is $u \in C^2(X)\cap C^1(\bar{X})$.
\end{enumerate}

\subsection{Discretisation of domains}
We begin by discretising the sets $X\subset\R$ and $Y\subset\R^2$.  Since $Y$ need only contain the support of the density $g$, we may assume that $Y$ is a square.

To discretise $X=[x_0,x_{max}]$, we choose an integer $N\in\N$ and introduce the grid spacing $h_X = (x_{max}-x_0)/N$.  Then grid points in the domain are given by $x_i = x_0+ih_X$ for $i = 0, \ldots, N$.

To discretise $Y = [y_0,y_{max}]^2$, we choose an integer $M\in\N$ and introduce the grid spacing $h_Y = (y_{max}-y_0)/M$.  Then grid points in the domain are given by $y_{jk} = (y_0+jh_Y,y_0+kh_Y)$ for $j,k=0,\ldots,M$.

We will require that asymptotically, $M\ll N$.  That is, the one-dimensional set $X$ needs to be more highly resolved than the two-dimensional set $Y$.

Throughout, we use the convention that the values of a function $u$ defined on the one-dimensional grid can be represented as $u_i = u(x_i)$.

\subsection{Discretisation of ODE}
We seek a finite difference discretisation of~\eqref{eq:MA1to2} of the form
\bq\label{eq:approx}
\begin{split}
F^{h_x,h_Y}_i&(u_i,u_i-u_{i-1},u_i-u_{i+1})\\ &= -h_Y^2\sum\limits_{j=0}^M\sum\limits_{k=0}^M \left(\Dt^{h_X}_{xx}u(x_i) + \frac{\partial^2c}{\partial x^2}(x_i,y_{jk})\right)^+\psi(x_i,y_{jk})\delta^{h_X,h_Y}(u_i-u_{i-1},u_i-u_{i+1})\\
&\phantom{=} + f(x_i) + (h_Y^2+h_X/h_Y) u_i
\end{split}
\eq
for $i = 1, \ldots, N-1$.
For $F^{h_X,h_Y}$ to be monotone, it must be a non-decreasing function of all of its arguments.  This, in turn, is guaranteed if each term appearing in the summation is non-negative and a non-increasing function of the differences $u_i-u_{i-1}$ and $u_i-u_{i+1}$.

The first term in the summation is easily approximated using standard centered differences via
\bq\label{eq:uxx}
\Dt_{xx}^{h_X}u_i = \frac{(u_{i-1}-u_i)+(u_{i+1}-u_i)}{h_X^2}.
\eq
This trivially has the required monotonicity properties, which are preserved in the operations of addition with a given grid function and restriction to the positive part.

Discretisation of the delta function, which is supported on a curve in $\R^2$, is much more delicate.  Indeed, naive extensions of typical discretisations of one-dimensional delta functions can lead to $\bO(1)$ errors in the approximation~\cite{tornbergSingular}.

We take as a starting point an approximate one-dimensional delta function defined in terms of the linear hat function via
\bq\label{eq:delta1}
\delta_\epsilon(z) = \frac{1}{\epsilon}\max\left\{1 - {\frac{\abs{z}}{\epsilon}},0\right\}.
\eq
Typically, the support of this discrete delta function is chosen to be proportional to the underlying grid spacing, that is $\epsilon \sim h_Y$.

It is tempting at this point to define the discrete delta function in~\eqref{eq:approx} by
\[ \delta^{h_X,h_Y}(u_i-u_{i-1},u_i-u_{i+1}) = \delta_{mh_Y}\left(\Dt_x^{h_X}u_i+\frac{\partial c}{\partial x}(x_i,y_{jk})\right) \]
where $\Dt_x^{h_X}u_i$ denotes some finite difference approximation of $u'(x_i)$ and $m\in\N$ is some fixed positive integer.  However, there are two problems with this: 
(1) fixing the value of $\epsilon$ in relation to the grid parameter $h_Y$ can lead to $\bO(1)$ errors~\cite{tornbergSingular} and (2) because of the nonlinearity, we cannot expect the overall approximation to be monotone even if $\Dt_x^{h_X}u_i$ is itself monotone.

Fortunately, the first problem can be resolved by allowing $\epsilon$ to vary in space~\cite{EngquistDelta}.  We begin by discretising the delta function in the $y$ variable only, allowing $x$ to be continuous.  Then we can introduce the approximate delta function
\bq\label{eq:deltaY}
\delta^{h_Y}\left(u'(x)+\frac{\partial c}{\partial x}(x,y_{jk})\right) = \delta_{\epsilon(h_Y,\nabla\phi^x(y))}\left(u'(x)+\frac{\partial c}{\partial x}(x,y_{jk})\right)
\eq
where the support parameter in the linear hat approximation to the delta function is given by
\bq\label{eq:epsilon}
\epsilon(h_Y,\nabla\phi^x(y)) = h_Y\max\{\|\nabla\phi^x(y)\|_{L^1(Y)},1\} = h_Y \max\left\{\| \nabla_y \frac{\partial c}{\partial x}(x,y)\|_{L^1(Y)},1\right\}.
\eq
We note that since the cost $c$ is a given function, this spatially varying parameter can be computed \emph{a priori}.

To achieve a monotone discretisation of this, we need to exploit the structure of the linear hat function.  To this end, we notice that
\begin{align*}
\delta_{\epsilon}&\left(u'(x)+\frac{\partial c}{\partial x}(x,y_{jk})\right) = \frac{1}{\epsilon}\max\left\{1 - {\frac{\abs{u'(x)+\frac{\partial c}{\partial x}(x,y_{jk})}}{\epsilon}},0\right\}\\
 & \phantom{===}=\frac{1}{\epsilon}\max\left\{1 - \frac{1}{\epsilon}{{\max\left\{u'(x)+\frac{\partial c}{\partial x}(x,y_{jk}),-u'(x)-\frac{\partial c}{\partial x}(x,y_{jk}),0\right\}}{}},0\right\}.
\end{align*}

This is automatically non-negative.  We can also produce a fully discrete approximation with the required monotonicity properties by utilising a careful combination of forward and backward differences.  This leads us to propose
\bq\label{eq:delta}
\begin{split}
\delta^{h_X,h_Y}&(u_i-u_{i-1},u_i-u_{i+1})\\ &= \frac{1}{\epsilon}\max\left\{1 - \frac{1}{\epsilon}{{\max\left\{\frac{u_i-u_{i-1}}{h_X}+\frac{\partial c}{\partial x}(x_i,y_{jk}),-\frac{u_{i+1}-u_i}{h_X}-\frac{\partial c}{\partial x}(x_i,y_{jk}),0\right\}}{}},0\right\}
\end{split}
\eq
with $\epsilon$ allowed to vary according to~\eqref{eq:epsilon}.

\subsection{Boundary conditions}
While Optimal Transport PDEs such as~\eqref{eq:localfirstMA} typically do not require a traditional boundary conditions\cite{Hamfeldt_BVP2}, the corresponding finite difference approximation~\eqref{eq:approx} will require some additional conditions to be supplied at the boundary grid points $x_0$ and $x_N$.


We have a great deal of freedom in the choice of condition to enforce at the remaining boundary point.  For the usual second boundary value problem for the \MA equation, it was shown in~\cite{Hamfeldt_BVP2} that even enforcing a possibly inconsistent boundary condition such as
\[ u(x) = 0, \quad x \in \partial X \]
would lead to a convergent numerical method.  However, this type of artificial boundary condition will result in a boundary layer in the computed solution.

To find a more natural boundary condition, we recall that the integrand in~\eqref{eq:MAlevel1} is supported on the set of points $y$ for which
\[ u'(x) + \frac{\partial c}{\partial x}(x,y) = 0. \]
In, particular, at the boundary point $x_i$ there must be some $y_i\in Y$ such that
\[ u'(x_i) = -\frac{\partial c}{\partial x}(x_i,y_i). \]
The cost function $c$ (and its partial derivatives) are all known functions.  However, the relevant values of $y$ are not obvious.  

We choose to proceed using a monotonicity assumption and enforce the Neumann boundary conditions
\begin{align*}
u'(x_0) &= \inf\limits_{y\in Y}\left\{-\frac{\partial c}{\partial x}(x_0,y)\right\},\\
u'(x_N) &= \sup\limits_{y\in Y}\left\{-\frac{\partial c}{\partial x}(x_N,y)\right\}.  
\end{align*}

These can easily be dicretised in a consistent, monotone way by defining the approximations
\bq\label{eq:bc2}
\begin{split}
F_0^{h_X,h_Y}(u_0,u_0-u_1) &= -\frac{u_1-u_0}{h_X} + \inf\limits_{y\in Y}\left\{-\frac{\partial c}{\partial x}(x_0,y)\right\} + h_X u_0,\\
F_N^{h_X,h_Y}(u_N,u_N-u_{N-1}) &= \frac{u_N-u_{N-1}}{h_X} - \sup\limits_{y\in Y}\left\{-\frac{\partial c}{\partial x}(x_0,y)\right\} + h_X u_N.
\end{split}
\eq

\subsection{Consistency and monotonicity}
We now verify the consistency and monotonicity of the approximation scheme described by~\eqref{eq:approx}-\eqref{eq:bc2}.  

First, we note that monotonicity was built into the approximation scheme by design.
\begin{lemma}[Monotonicity]\label{lem:mon}
Under the assumptions of Hypothesis~\ref{hyp}, the approximation scheme~\eqref{eq:approx}-\eqref{eq:bc2} is monotone.
\end{lemma}

Next, we verify consistency of the scheme.  Because of the interplay between the finite difference discretisations and the approximation of the Dirac measure, this is a more delicate calculation that requires a careful balance between the discretisation parameters $h_X$ and $h_Y$.

\begin{lemma}[Consistency]\label{lem:consistent}
Under the assumptions of Hypothesis~\ref{hyp}, consider the approximation scheme~\eqref{eq:approx}-\eqref{eq:bc2} where the grid parameters are chosen so that $h_X/h_Y\to 0$ as $h_Y\to 0$.  Then the scheme is consistent with the ODE~\eqref{eq:MA1to2}.  Moreover, the formal truncation error is given by $\bO(h_Y^2+h_X/h_Y)$.
\end{lemma}

Before proving this lemma, we introduce some notation designed to simplify the representation of the semi- and fully-discrete versions of the approximation.

Denote the coefficient of the delta distribution in~\eqref{eq:MA1to2} by
\bq\label{eq:G}
G(x,y) = \left(u''(x) + \frac{\partial^2c}{\partial x^2}(x,y)\right)^+\psi(x,y) 
\eq
and let
\bq\label{eq:Gh}
\begin{split}
G^{h_X}(x_i,y) &= \left(\Dt_{xx}^{h_X}u_i+ \frac{\partial^2c}{\partial x^2}(x_i,y)\right)^+\psi(x_i,y)\\
  &= G(x_i,y) + \bO(h_X^2)
\end{split}
\eq
be the discretised version.

We recall that the level set function used to locate the curve $Y_1(x,u'(x))$ is given by
\bq\label{eq:phi}
\phi^x(y) = u'(x) + \frac{\partial c}{\partial x}(x,y),
\eq
while its discretised version satisfies
\bq\label{eq:phih}
\begin{split}
\abs{\phi^{h_X}(x_i,y)} &= \max\left\{\frac{u_i-u_{i-1}}{h_X}+\frac{\partial c}{\partial x}(x,y_{jk}),-\frac{u_{i+1}-u_i}{h_X}-\frac{\partial c}{\partial x}(x,y_{jk}),0\right\}\\
  &= \abs{\phi^{h_X}_{x_i}(y)} + \bO(h_X).
\end{split}
\eq

Using this notation, the integral appearing in equation~\eqref{eq:MA1to2} takes the form
\bq\label{eq:int}
I(x) = \int_Y G(x,y)\delta(\phi^x(y))\,dy,
\eq
while its semi- and fully-discrete approximations are
\bq\label{eq:intsemi}
I^{h_Y}(x) = h_Y^2\sum\limits_{j=0}^M\sum\limits_{k=0}^M \frac{1}{\epsilon}\max\left\{1-\frac{\abs{\phi^x(y_{jk})}}{\epsilon},0\right\} G(x,y_{jk}),
\eq
\bq\label{eq:intdisc}
I^{h_X,h_Y}(x_i) = h_Y^2\sum\limits_{j=0}^M\sum\limits_{k=0}^M \frac{1}{\epsilon}\max\left\{1-\frac{\abs{\phi^{h_X}({x_i},y_{jk})}}{\epsilon},0\right\} G^{h_X}(x_i,y_{jk}).
\eq

\begin{proof}[Proof of Lemma~\ref{lem:consistent}]
We begin by noting that the width of the support of the approximate delta function can be bounded in terms of the spatial discretisation parameter $h_Y$.  Let
\[ K = {\sup\limits_{x\in X} \| \nabla_y \frac{\partial c}{\partial x}(x,y) \|_{L^1(Y)}}.\]
Then $\epsilon(h_X,\nabla\phi^x(y)) \leq Kh_Y$ for every $x, y$.

We also notice that many of the terms in the sum of~\eqref{eq:intsemi} are zero.  To identify the non-zero terms, we define
\bq\label{eq:nonzero_semi}
\begin{split}
\I^{h_Y}(x) &= \{(j,k) \mid \abs{\phi^x(y_{jk})}<\epsilon\}\\
 &\subseteq \{(j,k) \mid \abs{\phi^x(y_{jk})}<Kh_Y\}.
\end{split}
\eq
Since $\phi^x(y)$ is uniformly continuous in both $x$ and $y$ and its zero level-set defines a curve in $Y$, the cardinality of this set will be
\[ \abs{\I^{h_Y}} = \bO(1/h_Y). \]

We can also try to identify the terms that are non-zero in the fully discrete sums in~\eqref{eq:intdisc}.  To this end, we set
\bq\label{eq:nonzero_disc}
\begin{split}
\I^{h_X,h_Y}(x_i) &= \{(j,k) \mid \abs{\phi^{h_X}(x_i,y_{jk})}<\epsilon\}\\
 &\subseteq \{(j,k) \mid \abs{\phi^{x_i}(y_{jk}) + \bO(h_X)}<Kh_Y \}.
\end{split}
\eq
Any point $(j,k)$ in this set will satisfy 
\[ \abs{y_{jk}-y} < \bO(h_X + h_Y) \]
for some $y \in \I^{h_Y}(x_i)$.  This means that for each such $y$, we need to consider $\bO\left(\dfrac{h_X+h_Y}{h_Y}\right) = \bO(1)$ possible neighbouring grid points that satisfy the conditions in~\eqref{eq:nonzero_disc}.  Adding up all possible points $y\in I^{h_Y}(x_i)$ leads to the bound
\[ \abs{\I^{h_X,h_Y}(x_i)} = \bO(1/h_Y).\]

Using the results of~\cite{EngquistDelta}, we know that the discretisation of the Dirac delta function satisfies
\[ I^{h_Y}(x) = I(x) + \bO(h_Y^2) \]
since the integrand $G(x,y)$ is $C^2$ in the $y$-variable. 

Finally, we can consider the approximation accuracy of the fully discrete version of the integral~\eqref{eq:intdisc}.
\begin{align*}
I^{h_X,h_Y}(x_i) &= h_Y^2 \sum\limits_{j=0}^M\sum\limits_{k=0}^M\frac{1}{\epsilon}\max\left\{1-\frac{\abs{\phi^{h_X}({x_i},y_{jk})}}{\epsilon},0\right\} G^{h_X}(x_i,y_{jk})\\
  &= h_Y^2 \sum\limits_{j=0}^M\sum\limits_{k=0}^M\frac{1}{\epsilon}\max\left\{1-\frac{\abs{\phi^{x_i}(y_{jk}) + \bO(h_X)}}{\epsilon},0\right\} G^{h_X}(x_i,y_{jk})\\
	&= h_Y^2\sum\limits_{j=0}^M\sum\limits_{k=0}^M\frac{1}{\epsilon}\max\left\{1-\frac{\abs{\phi^{x_i}(y_{jk}) }}{\epsilon},0\right\} G^{h_X}(x_i,y_{jk})\\ &\phantom{=} +\bO\left(\frac{h_Y^2}{\epsilon}\sum\limits_{(j,k)\in\I^{h_Y}(x_i) \cup\I^{h_X,h_Y}(x_i)}\frac{h_X}{\epsilon}G^{h_X}(x_i,y_{jk})\right).
\end{align*}
In the last step, we notice that the error term will only be non-zero at indices where the summand is non-zero in either the semi- or fully-discrete approximation of~\eqref{eq:intsemi}-\eqref{eq:intdisc}.

We can further simplify this to
\begin{align*}
I^{h_X,h_Y}(x_i) &=h_Y^2\sum\limits_{j=0}^M\sum\limits_{k=0}^M\frac{1}{\epsilon}\max\left\{1-\frac{\abs{\phi^{x_i}(y_{jk}) }}{\epsilon},0\right\} (G(x_i,y_{jk}) + \bO(h_X^2))\\&\phantom{=} + \bO\left(\frac{h_Y^2h_X}{\epsilon^2}\abs{\I^{h_Y}(x_i) \cup\I^{h_X,h_Y}(x_i)}\right)\\
&= I^{h_Y}(x_i) + \bO\left(\frac{h_X^2h_Y^2}{\epsilon}\abs{\I^{h_Y}(x_i)}\right) + \bO\left(\frac{h_Y^2h_X}{\epsilon^2}\abs{\I^{h_Y}(x_i) \cup\I^{h_X,h_Y}(x_i)}\right)\\
&= I(x_i) + \bO\left(h_Y^2 + \frac{h_X^2h_Y}{\epsilon} + \frac{h_Xh_Y}{\epsilon^2}\right).
\end{align*}

Finally, we recall that $h_Y \leq \epsilon = \bO(h_Y)$ so that
\[ I^{h_X,h_Y}(x_i) = I(x_i) + \bO\left(h_Y^2 + h_X^2 + \frac{h_X}{h_Y}\right). \]
By hypothesis, $h_X \ll h_Y$ so that the approximation of the integral is consistent with a formal truncation error of $\bO(h_Y^2 + h_X/h_Y)$.  Since the remaining terms in the approximation~\eqref{eq:approx} share the same consistency error, this completes the proof.
\end{proof}

We also make the remark here that the formal consistency error in the scheme can achieve an optimal value of $\bO(h_X^{2/3})$ if we make the scaling choice $h_Y = \bO(h_X^{1/3})$.  In particular, this means that the two-dimensional set $Y$ can (and should) have a much lower resolution than the one-dimensional set $X$.  We also note that typical convergence analysis for \MA equations in optimal transport involves compactness arguments without error bounds; actual computational error need not coincide with the formal truncation error.

\subsection{Normalisation}
Finally, we notice that solutions of the ODE~\eqref{eq:MA1to2} are only determined up to additive constants.  The particular constant is not typically important to applications, as the information about the transport mapping is encoded in $u'(x)$.

For the purposes of this article, we will seek to approximate the mean-zero solution of~\eqref{eq:MA1to2}.  There is no particular reason that the approximate solution we obtain from~\eqref{eq:approx}-\eqref{eq:bc2} should have mean zero, or even that the mean need be the same given different values of the discretisation parameters $h_X, h_Y$.  This is typical of \MA equations in optimal transport, and an easy solution is to simply shift the approximate solution~\cite{HT_OTonSphereTheory}.

Thus we propose the following two-step procedure:
\begin{enumerate}
\item Let $v^{h_X,h_Y}$ be the unique solution of
\[ F^{h_X,h_Y}_i(v_i^{h_X,h_Y}(x),v_i^{h_X,h_Y}(x)-v^{h_X,h_Y}_{i-1},v_i^{h_X,h_Y}(x)-v^{h_X,h_Y}_{i+1}) = 0 \]
for all $i = 0, \ldots, N$.
\item Normalise to
\[ u^{h_X,h_Y}_i = v^{h_X,h_Y}_i - \frac{1}{N}\sum\limits_{j=1}^Nv^{h_X,h_Y}_j.\]
\end{enumerate}

\section{Computational Results}\label{sec:results}
Finally, we validate our proposed numerical method on several computational examples.

In each example, we take as our domain $X = [0,1]$, which is discretised using $N$ grid points.  The computational target set is taken to be the smallest possible square that encloses the actual target $Y$.

The grid spacing parameters are chosen to satisfy $h_Y \approx h_X^{1/3}$, with slight deviations from equality used to ensure that this leads to an integer $M$ number of grid points along each dimension of $Y$.

The resulting systems of nonlinear equations were solved using Newton's method.

\subsection{Rectangular target}
The first example we consider involves a rectangular domain
\[ Y = \{(y_1,y_2) \mid 1 \leq y_1 + y_2 \leq e+2, y_1-1 \leq y_2 \leq y_1+1\}. \]
Note that this domain is rotated, so it will not line up with the grid used to discretise the $y$ variables.

We solve~\eqref{eq:ot} using the density functions
\begin{align*} f(x) &= (e^x+2)(e^x+2x),\\  g(y_1,y_2) &= \begin{cases}y_1 + y_2, & (y_1,y_2) \in Y \\ 0, & \text{otherwise}\end{cases}\end{align*}
together with the quadratic cost function
\[ c(x,y_1,y_2) = \frac{1}{2}(x-y_1)^2 + \frac{1}{2}(x-y_2)^2. \]

See Figure~\ref{fig:ex2data} for a visual of the problem data.

\begin{figure}[htp]
\centering
\subfigure[]{
\includegraphics[width=0.45\textwidth]{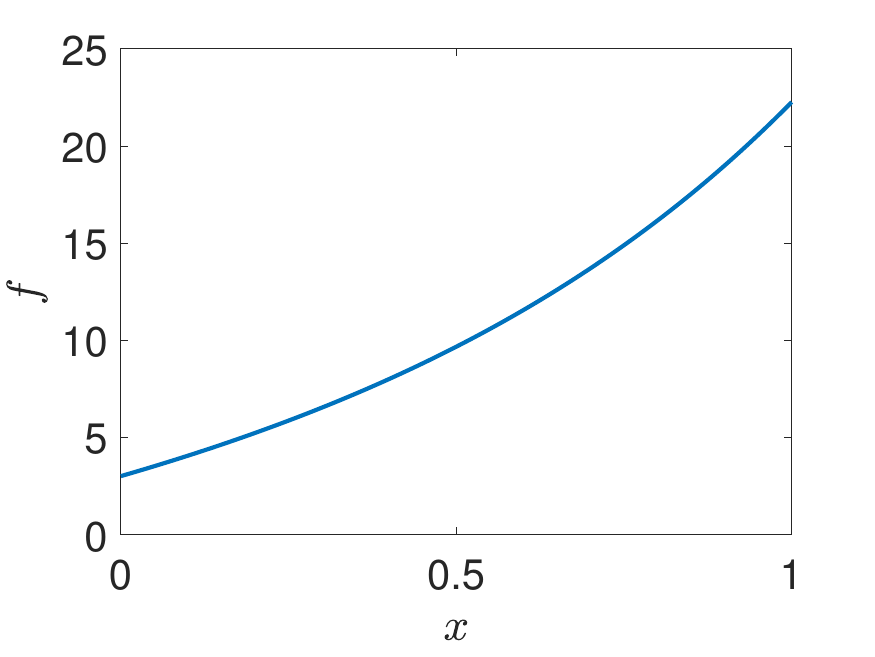}\label{fig:ex2f}}
\subfigure[]{
\includegraphics[width=0.45\textwidth]{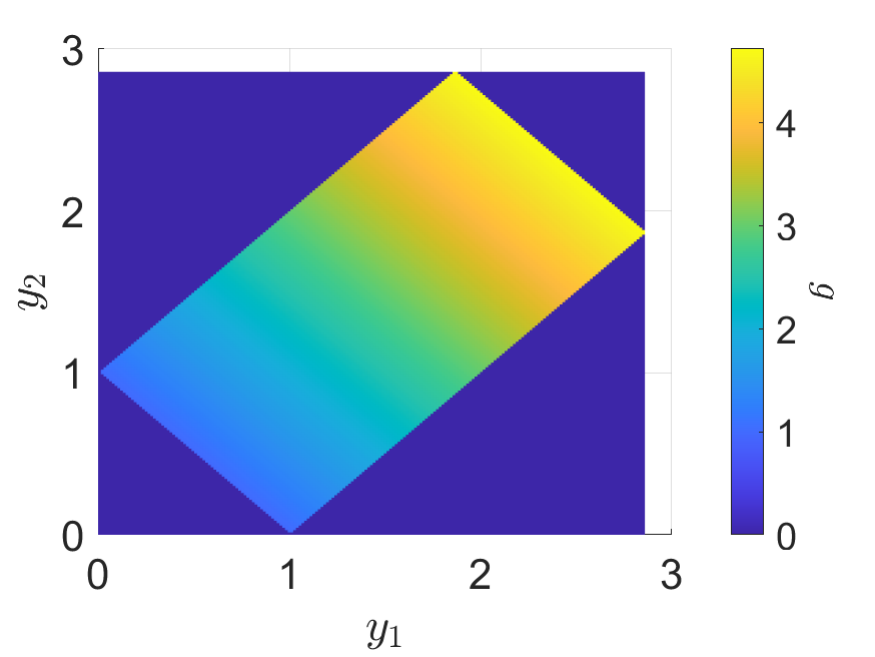}\label{fig:ex2g}}
\caption{The density functions \subref{fig:ex2f}~$f$ and \subref{fig:ex2g}~$g$ for the rectangular target.}
\label{fig:ex2data}
\end{figure}

We can easily verify that the exact mean-zero solution of the resulting \MA equation is
\[ u(x) =  e^x - e+1.\]

The resulting error is displayed in Figure~\ref{fig:ex2err}.  We notice that for this particular example, the error decreases, but not monotonically.  It is well known that monotone finite difference schemes do not have to yield monotonic convergence.  In this case, it appears to be the simplicity of the geometry that leads to the oscillations in error.  This is because for some choices of discretisation parameters $M, h_Y$, the grid is by chance quite well-aligned with the relevant curve $Y_1(x,u'(x))$, the contours of $g$, and the boundaries of the target $Y$. This, in turn, can lead to artificially low error for certain values of the discretisation parameters.  

\begin{figure}[htp]
\centering
\includegraphics[width = 0.6\textwidth]{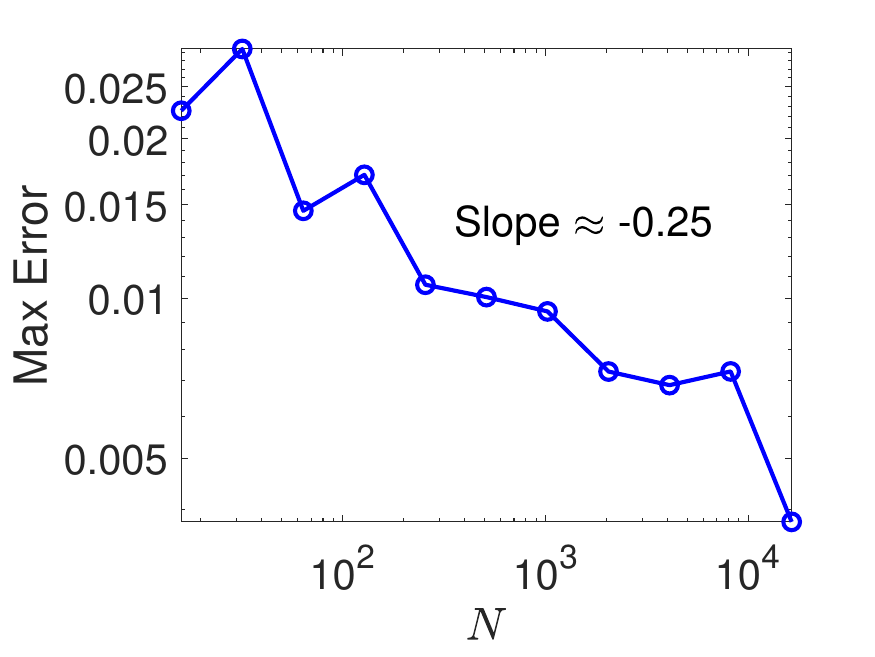}
\caption{Maximum error in $u$ for the rectangular target.}
\label{fig:ex2err}
\end{figure}

\subsection{Vanishing densities}
The next example involves another rectangular domain
\[ Y = \{(y_1,y_2) \mid 0 \leq y_1 + y_2 \leq 2, y_1-2 \leq y_2 \leq y_1+2\} \]
and the quadratic cost
\[ c(x,y_1,y_2) = \frac{1}{2}(x-y_1)^2 + \frac{1}{2}(x-y_2)^2 \]
as in the previous example.  This time, we use density functions that are allowed to vanish to zero at the boundaries of $X$ and $Y$.  This setting is known to lead to a loss of uniform ellipticity in the \MA equation, which is often a challenge for numerical methods. We choose the particular density functions
\begin{align*}
f(x) &= \frac{128}{3}x(1-x),\\
g(y_1,y_2) &= \begin{cases}
(y_1+y_2)\left(2-(y_1+y_2)\right)\left((y_2-y_1)-2\right)\left(2-(y_2-y_1)\right), & (y_1,y_2) \in Y\\
0, & \text{otherwise.}
\end{cases}
\end{align*}

See Figure~\ref{fig:ex4data} for a visual of the problem data.   

\begin{figure}[htp]
\centering
\subfigure[]{
\includegraphics[width=0.45\textwidth]{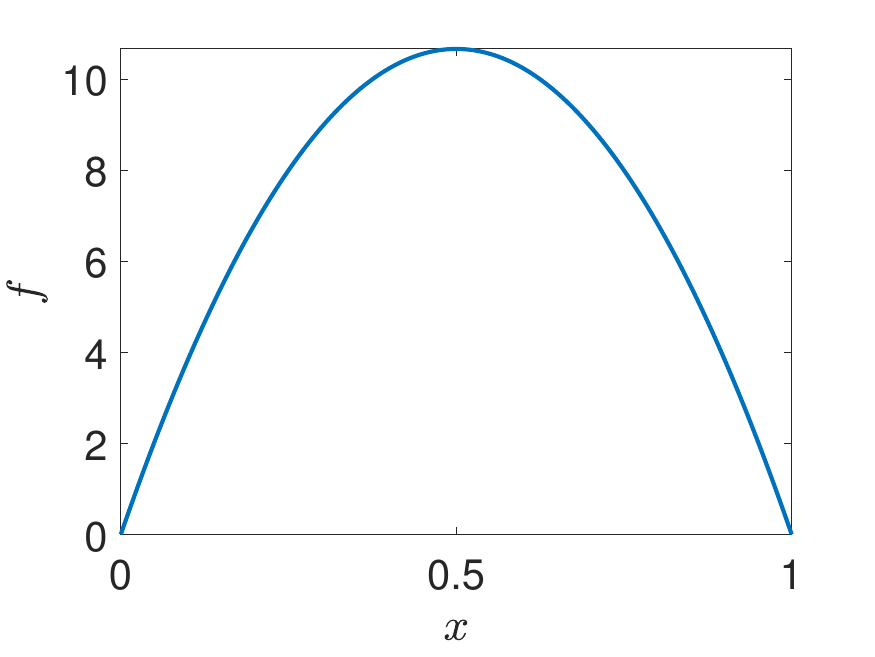}\label{fig:ex4f}}
\subfigure[]{
\includegraphics[width=0.45\textwidth]{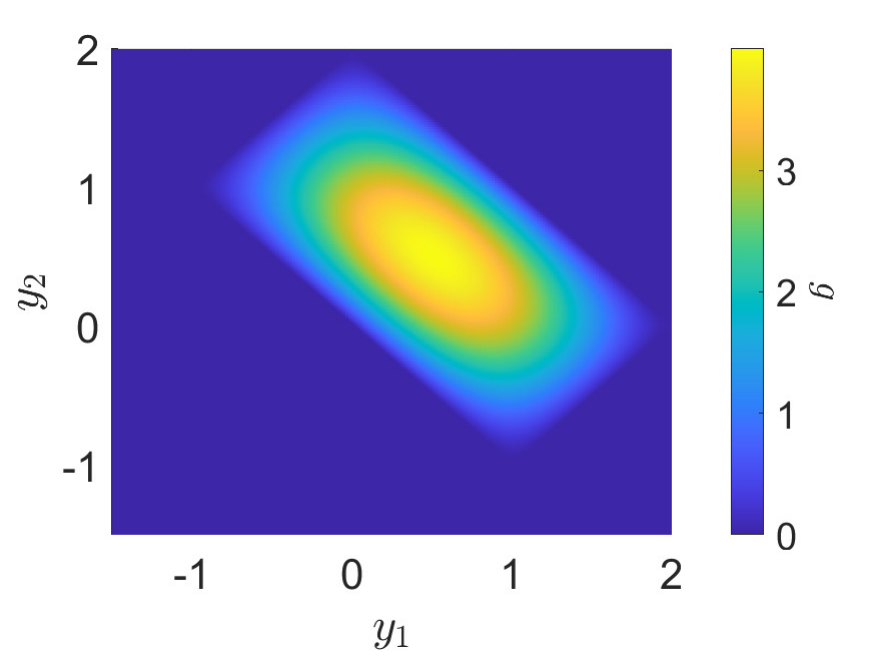}\label{fig:ex4g}}
\caption{The density functions \subref{fig:ex4f}~$f$ and \subref{fig:ex4g}~$g$ for the rectangular target.}
\label{fig:ex4data}
\end{figure}

In this case, the exact solution is the constant function
\[ u(x) =  0.\]

Although the exact solution is constant, we cannot expect to recover this exactly (which would occur with a typical finite difference discretisation of a standard \MA equation).  This is because the approximation scheme will still involve (approximate) integration of a spatially varying function over a non-trivial curve, which will introduce numerical errors.

The error obtained for this example is displayed in Figure~\ref{fig:ex4err}.  In this example, the error appears to decrease monotonically, likely because the contours of $g$ are no longer straight lines that can align with the grid.  Despite the degeneracy of this example, there is no difficult in computing this solution. 

\begin{figure}[htp]
\centering
\includegraphics[width = 0.6\textwidth]{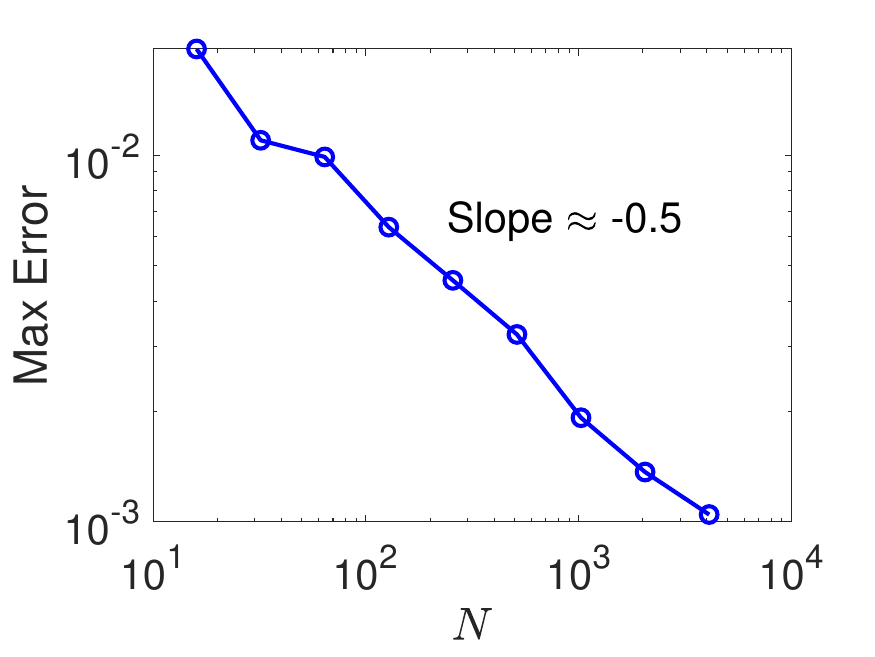}
\caption{Maximum error in $u$ for the vanishing densities.}
\label{fig:ex4err}
\end{figure}

\subsection{Curved target}
The final example we consider involves the curved domain
\[ Y = \left\{(y_1,y_2) \mid -1 \leq y_1 \leq 1, -\frac{\pi}{2} \leq y_2 - \sqrt{1-y_1^2} \leq 0\right\}. \]

We solve~\eqref{eq:ot} using the density functions
\begin{align*} f(x) &= \frac{\pi^3}{8}\sin(\pi x),\\  g(y_1,y_2) &= \begin{cases}\sqrt{1-y_1^2}-y_2, & (y_1,y_2) \in Y \\ 0, & \text{otherwise}\end{cases}.\end{align*}
These densities are also allowed to vanish along part of the boundary.  In this example, we use the cost function
\[ c(x,y_1,y_2) = x\left(y_2-\sqrt{1-y_1^2}\right). \]

See Figure~\ref{fig:ex3data} for a visual of the problem data.

\begin{figure}[htp]
\centering
\subfigure[]{
\includegraphics[width=0.45\textwidth]{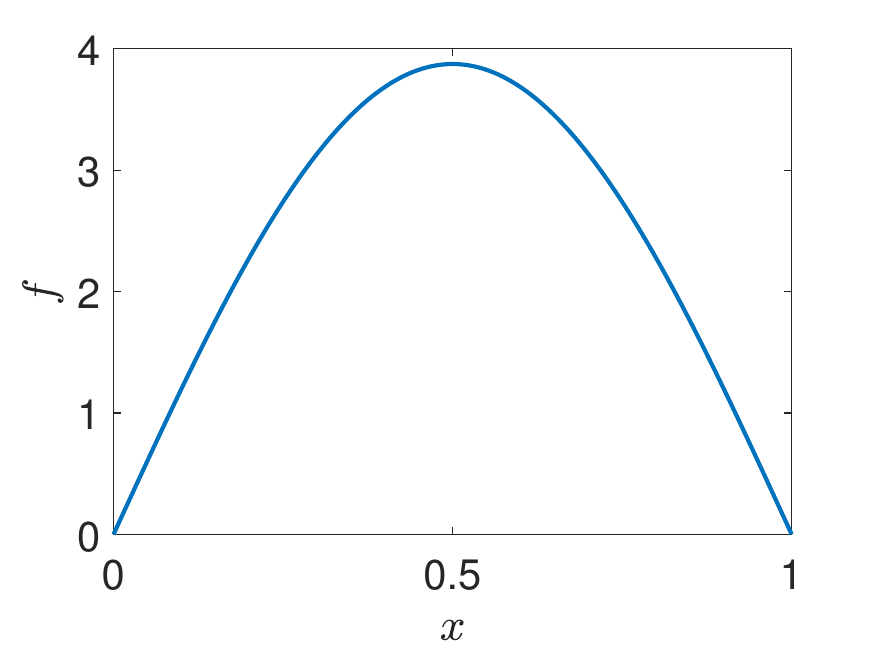}\label{fig:ex3f}}
\subfigure[]{
\includegraphics[width=0.45\textwidth]{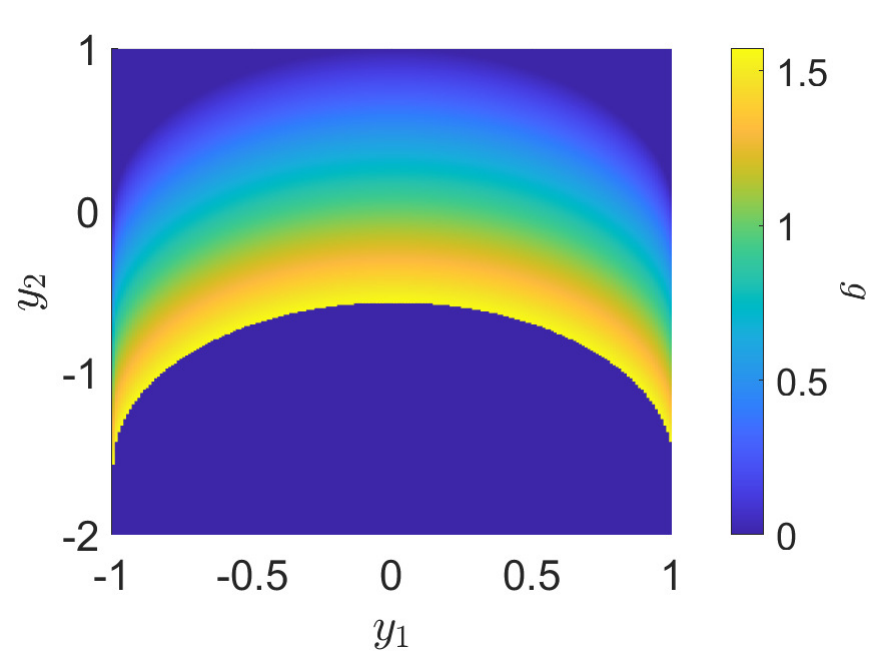}\label{fig:ex3g}}
\caption{The density functions \subref{fig:ex3f}~$f$ and \subref{fig:ex3g}~$g$ for the rectangular target.}
\label{fig:ex3data}
\end{figure}

The exact mean-zero solution of the resulting \MA equation is
\[ u(x) = -\cos\left(\frac{\pi x}{2}\right) + \frac{2}{\pi}.\]

The resulting error is displayed in Figure~\ref{fig:ex3err}.  Once again, we observe clear monotonic convergence.  The more complicated geometry does not pose any difficulty for the numerical method.

\begin{figure}[htp]
\centering
\includegraphics[width = 0.6\textwidth]{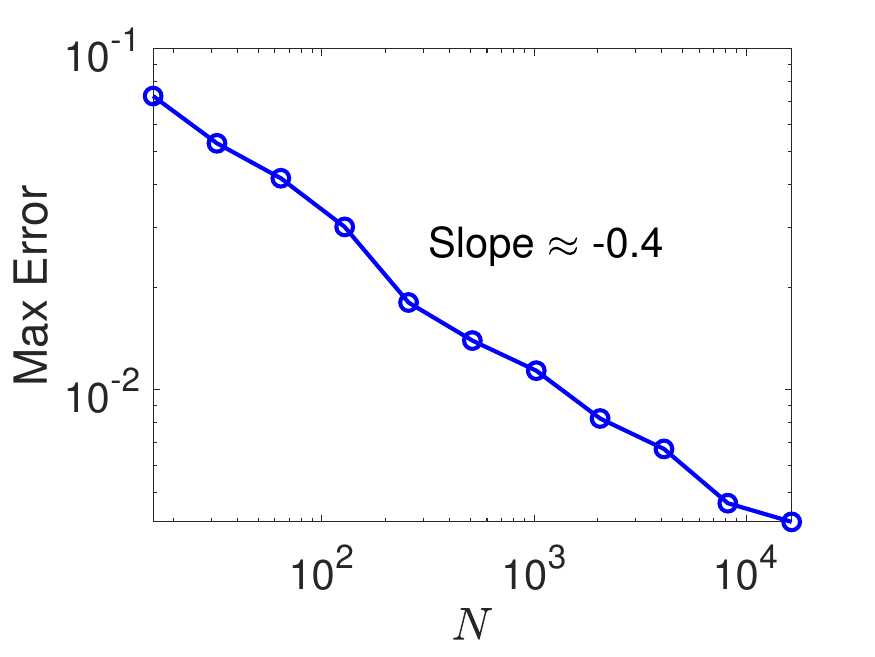}
\caption{Maximum error in $u$ for the curved target.}
\label{fig:ex3err}
\end{figure}

\section{Conclusions}\label{sec:conclusions}
In this article, we have proposed a numerical method for a \MA type equation arising in the problem of optimal transport between unequal dimensions.  In additional to the strong nonlinearity of this problem, the equation involves an extra layer of complexity because it requires integrating the equation over a manifold that itself depends on the solution of the \MA equation.

We have proposed a level-set formulation of this PDE, which allows the integral to be posed over a much simpler domain.  However, the resulting equation still involves (1) a fully nonlinear elliptic operator of \MA type and (2) integration against a Dirac delta distribution whose support is dependent on the solution gradient.

We have proposed a very novel method for numerically approximating the resulting equation.  For simplicity, we focus on the setting of optimal transport from 1D to 2D.  However, all of the techniques we proposed can be easily extended to higher dimensions.  The numerical method utilises monotone finite difference approximations of the \MA type operator.  These are combined with a very careful discrete version of the Dirac delta function, which is allowed to have a spatially varying support width in order to preserve consistency.  The level set function that is input into this discrete delta function is discretised using a careful combination of forward and backward differences in order to preserve monotonicity.

The resulting numerical method is both consistent and monotone.  For more well-understood \MA type equations in optimal transport, consistency and monotonicity have proven to be the key ingredients in proofs of convergence~\cite{bonnet2022monotone,Hamfeldt_BVP2,HT_OTonSphereTheory}.  With ongoing development of the theory for the non-standard \MA equation considered in this article, there is great reason to hope that a convergence proof can be constructed for the method proposed here.  This will be an avenue of continued research.

We have performed computational tests using several challenging examples, which validate the performance of the new method.

A natural avenue for future work is to extend this numerical method to higher dimensions.  An advantage of the current method is that the higher-dimensional set $Y$ can be discretised using a fairly coarse grid, which is important for keeping the computational cost reasonable.  Future work will further improve the computational cost by devising better strategies to estimate the support of the discrete delta function, allowing for a great reduction in the cost of the numerical quadrature.  We expect the resulting method to be easy to implement, while having a computational cost equivalent to integrating over an $m-n$-dimensional manifold (instead of the full $m$-dimensional target set) at each point in the lower-dimensional domain $X$.

\bibliographystyle{plain}
\bibliography{OT1Dto2D}

\end{document}